\documentclass[a4paper,12pt]{article}
\usepackage[top=2.5cm,bottom=2.5cm,left=2.5cm,right=2.5cm]{geometry}
\usepackage[margin=1cm,%
font=small,%
format=hang,%
labelsep=period,%
labelfont=bf]{caption}
\usepackage[utf8]{inputenc}
\usepackage{amsmath}
\usepackage{amsfonts}
\usepackage{amssymb}
\usepackage{cite}
\usepackage{amsthm}
\usepackage{makeidx}
\usepackage{graphicx}
\usepackage{mathrsfs,xcolor}
\usepackage{blkarray}
\usepackage{bm}
\usepackage{setspace}
\usepackage[left,pagewise,displaymath, mathlines]{lineno}
\usepackage{multirow}
\usepackage{hhline}
\usepackage[numbers]{natbib}

\usepackage{amsthm}%
\usepackage{mathrsfs}%
\usepackage[title]{appendix}%
\usepackage{listings}%
\usepackage{hhline}
\usepackage{enumerate}
\usepackage{bm}
\usepackage{blkarray}
\usepackage{kbordermatrix}
\usepackage{xfrac}

\makeatletter
\newcommand\xleftrightarrow[2][]{\ext@arrow 0099{\longleftrightarrowfill@}{#1}{#2}}
\def\longleftrightarrowfill@{\arrowfill@\leftarrow\relbar\rightarrow}
\makeatother

\numberwithin{equation}{section}

\theoremstyle{plain}
\newtheorem{theorem}{Theorem}[section]

\newtheorem{proposition}{Proposition}[section]

\theoremstyle{definition}
\newtheorem{definition}{Definition}[section]

\newtheorem{remark}{Remark}[section]

\usepackage{authblk}

\author[1,*]{ \textbf{Noel T. Fortun}}
\author[1,2,3]{ \textbf{Angelyn R. Lao}}
\author[1,2,4]{\textbf{Eduardo R. Mendoza}}
\author[5]{\textbf{Luis F. Razon}}

\affil[1]{\footnotesize \textit{Department of Mathematics and Statistics, De La Salle University, Manila  0922, Philippines}}
\affil[2]{\textit{Center for Natural Sciences and Environmental Research, De La Salle University, Manila  0922, Philippines}}
\affil[3]{\textit{Center for Complexity and Emerging Technologies, De La Salle University, Manila  0922, Philippines}}
\affil[4]{\textit{Max Planck Institute of Biochemistry, Martinsried near Munich, Germany}}
\affil[5]{\textit{Department of Chemical Engineering, De La Salle University, Manila  0922, Philippines}}

\affil[*]{Corresponding author: \texttt{noel.fortun@dlsu.edu.ph}}

\title{\textbf{Parameter-minimal analysis of carbon dioxide removal through direct air capture}}

\date{}

\begin{document}
\maketitle
\thispagestyle{empty}
\begin{abstract}
The potential for multistationarity, or the existence of steady-state multiplicity, in the Earth System raises concerns that the planet could reach a climatic `tipping point,' rapidly transitioning to a warmer steady-state from which recovery may be practically unattainable. In detailed Earth models that require extensive computation time, it is difficult to make an a priori prediction of the possibility of multistationarity. In this study, we demonstrate Chemical Reaction Network Theory (CRNT) analysis of a simple heuristic box model of the Earth System carbon cycle with the human intervention of Direct Air Capture. CRNT leverages parameter-minimal analysis, relying primarily on the graphical and kinetic structure of the reaction network system, to identify necessary conditions for steady-state multiplicity. The analysis reveals necessary conditions for the combination of system parameters where steady-state multiplicity may exist. With this method, other negative emissions technologies (NET) may be screened in a relatively simple manner to aid in the priority setting by policymakers. Beyond multistationarity, the analysis provides insights into key system properties, such as absolute concentration robustness and some conditions for atmospheric carbon reduction.
\end{abstract}

\singlespacing

\section{Introduction}

Carbon dioxide removal (CDR) technologies play a crucial role in mitigating climate change by removing CO$_2$ from the atmosphere. These technologies encompass a variety of methods such as afforestation, reforestation, direct air capture, and bioenergy with carbon capture and storage \cite{GALAN2021,KERNER2023}. Implementing CDR technologies is vital in achieving the goals set out in the Paris Agreement, which aims to keep the global mean surface temperature well below 2°C and target 1.5°C \cite{CLIMCH2022}. 

Specifically, Direct Air Capture (DAC) represents a state-of-the-art solution in the fight against climate change.  By capturing carbon dioxide directly from the atmosphere, DAC may help reduce the levels of this greenhouse gas, thereby lessening the impact of global warming. In this technology, carbon dioxide is directly captured from the atmosphere using chemical absorbents. The captured carbon is securely stored in geological formations, preventing its release back into the atmosphere \cite{KERNER2023,QIU2022}. 

As researchers explore the intricate dynamics of climate change, the notion of climate tipping points has emerged as a key area of interest \cite{DAKOS2024,foley2005tipping,AND2013,ANDERIES2023}. These tipping points represent moments when the climate system reaches a threshold and undergoes self-perpetuating changes, leading to profound and possibly irreversible impacts on our planet. Understanding and anticipating these tipping points is essential to formulate successful approaches for lessening the impact of climate change.  Researchers have developed sophisticated models, primarily utilizing numerical simulations, to simulate the intricate dynamics of multistationarity in the global carbon system
\cite{DAKOS2024}. 
However, the challenge lies in pinpointing the precise conditions that can trigger multistationarity within the system \cite{QIU2022,REALM2019,LEHTVEER2021,ANDERIES2023}.

In this study, we explore the possibility of the global carbon system, under DAC intervention, displaying multiple steady states by employing a methodology based in reaction networks. The first step entails building a reaction network that mirrors the dynamic behavior and properties of the global carbon system under investigation. By applying chemical reaction network theory (CRNT), key characteristics, including the system's potential for multiple steady states, are efficiently identified. In addition, the analysis offers valuable insights into essential system characteristics, such as absolute concentration robustness, along with specific conditions for reduction of atmospheric carbon.

CRNT provides a unique advantage by focusing on the topological features and kinetics of the network itself, eliminating the need to define system rate constants. This characteristic is especially beneficial when studying systems where such parameter values are unknown. By enabling a rate-constant-minimal analysis, CRNT promises to be a useful tool for revealing the intricacies of systems with uncertain rate-constant data, thereby providing critical insights into the dynamic behavior of the global carbon system.

\section{Model and Methods}

The global carbon cycle has been modeled with varying complexity. While sophisticated state-of-the-art models offer highly accurate Earth system predictions, they rely on complex networks (with an enormous number of equations) that are computationally expensive. Here, we focus on a simple heuristic box model, based on the spatial aggregation of carbon while focusing only on the processes that are most significant at a global scale. Though box models of carbon cycle are not quantitatively precise, they provide reasonable accuracy and help reveal important interactions and feedback often obscured in more complex systems \cite{SCHM2002,AND2013, HECK2016}.

\subsection{The model}

\begin{figure}[t]%
\centering
\includegraphics[width=0.8\textwidth]{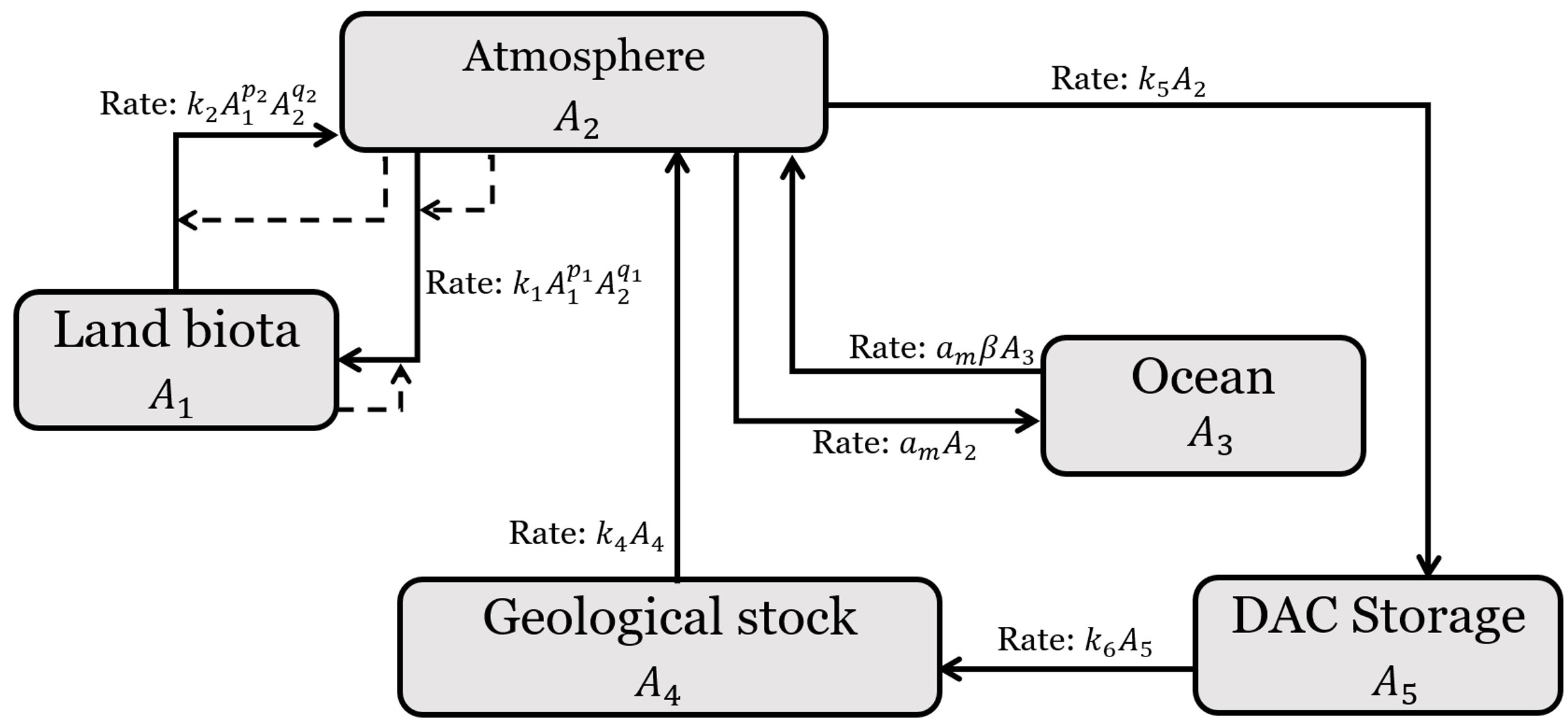}
\caption{In the box model, the boxes represent the different pools, solid arrows indicate the transfer of carbon from one pool to another, and dashed arrows indicate the pools that influence a carbon transfer.}\label{fig:DACPL}
\end{figure}

The pre-industrial system of \citet{AND2013} forms the building block for developing and examining the global carbon cycle system with DAC intervention. The model extends the initial three-box model that considers carbon interactions in the land-atmosphere-ocean system of \citet{AND2013}, denoted by $A_1$, $A_2$ and $A_3$ respectively. The schematic diagram is shown in Figure \ref{fig:DACPL}. The modeling framework relies on ordinary differential equations (ODEs) where all processes are modeled by products of power law functions. More precisely, a Generalized Mass Action (GMA)  system is an ODE system established by individually approximating each process in the system with a power-law term \cite{SAVA1998,SAVA1969,VOIT2000,VOIT2006,VOIT2013,TORRES2002}. These terms are then aggregated, with incoming fluxes indicated by a plus sign and outgoing fluxes by a minus sign. The procedure for deriving power-law approximations of rate functions is based on Taylor linearization in logarithmic coordinates. For a detailed derivation of the power-law approximation of the process rates in the pre-industrial carbon cycle model developed by \citet{AND2013}, refer to Appendix B of \cite{DOA2018}.

The current model incorporates the industrial carbon transfer activities (such as fossil fuel combustion) that lead to the linear transfer of carbon geological stock ($A_4$) to the atmosphere. DAC intervention is introduced by incorporating an extra box to store carbon ($A_5$) directly sequestered from the atmosphere. The rate of transfer is also assumed to be linear. The system also introduces a possible leak, which can be used to assess the CDR performance of the system even in the presence of such a leak.

Together with the power-law approximation of the transfer rates in the pre-industrial model and the linear functions of carbon transfer from $A_4$ to $A_2$, and $A_5$ to $A_4$, we form the system of ODEs of the carbon cycle system with DAC intervention: 
\begin{equation}\label{eq:DAC}
\left.
 \begin{array}{cl}
\dot{A_1}&=k_1A_1^{p_1}A_2^{q_1} - k_2A_1^{p_2}A_2^{q_2}  \\
\dot{A_2 }&= k_2A_1^{p_2}A_2^{q_2} - k_1A_1^{p_1}A_2^{q_1} -a_mA_2 +a_m\beta A_3+k_4A_4-k_5A_2\\
\dot{A_3} &=a_mA_2 -a_m\beta A_3 \\
\dot{A_4} &= k_6A_5 - k_4A_4\\
\dot{A_5} &= k_5A_2 - k_6A_5
 \end{array}
 \right.
\end{equation} 

Table \ref{tab1}  provides a summary of the system parameters that will be referenced in the model's specification and analysis. 

\begin{table}
\begin{center}
\begin{minipage}{\textwidth}
\caption{Model parameters}\label{tab1}
\begin{tabular}{|c|l|}
\hline
\textbf{Symbol} & \textbf{Description}  \\ 
\hline
\multirow{3}{*}{$p_1$ and $p_2$}	&	Kinetic orders of land interaction		\\
\hhline{~}   &  $p_1$: Kinetic order of land photosynthesis interaction ($p$-interaction)  \\
\hhline{~}  &  $p_2$: Kinetic order of land respiration interaction ($r$-interaction) 	 \\
\hline
\multirow{3}{*}{$q_1$ and $q_2$}	&	Kinetic orders of atmosphere interaction		\\
\hhline{~}   &  $q_1$: Kinetic order of atmosphere photosynthesis interaction ($p$-interaction)  \\
\hhline{~}  &  $q_2$: Kinetic order of atmosphere respiration interaction ($r$-interaction) 	 \\
\hline
\multirow{1}{*}{$p_2-p_1$}	&	Land $r$-$p$-interaction difference	\\
\hline
\multirow{1}{*}{$q_2-q_1$}	&	Atmosphere $r$-$p$-interaction difference	\\
\hline
\multirow{3}{*}{$\displaystyle R_p=\frac{p_2-p_1}{q_2-q_1}$}	&		\\
\hhline{~}  &   Land-atmosphere $r$-$p$-intearction difference ratio	 \\
\hhline{~}  &   	 \\
\hline
\multirow{3}{*}{$\displaystyle R_q=\frac{q_2-q_1}{p_2-p_1}$}	&		\\
\hhline{~}  &   Atmosphere-land $r$-$p$-intearction difference ratio	 \\
\hhline{~}  &   	 \\
\hline
\end{tabular}
\end{minipage}
\end{center}
\end{table}

\subsection{Power-law kinetic representation of an Earth system model}

The analysis of the model involves constructing a \textbf{power-law kinetic representation}. Such a representation involves a chemical reaction network (CRN) with power-law kinetics that is dynamically equivalent to the original system (i.e., they share the same ODEs). The goal is to identify important system dynamic properties through its power-law kinetic representation without the need for numerical computations or simulations typically required for nonlinear ODEs. Leveraging tools and insights from CRNT, the analysis is conducted with minimal reliance on specific parameters, as it treats rate constants symbolically and avoids dependence on their numerical values. 

In the following section, we outline the process of constructing the CRN network for the model of interest and defining the power-law kinetics for this CRN.

\subsubsection{Chemical reaction network representation}

We outline some foundational concepts pertaining to CRNs. While these fundamentals are essential, they are not exhaustive and formally presented, especially in the context of the mathematical proofs discussed in this paper. For a formal treatment of the relevant concepts, notations, and results in CRNT necessary for proving our findings, readers are referred to Appendix \ref{appendix: CRNT}.

\subsubsection*{Chemical reaction network preliminaries}

A chemical reaction network or CRN is a finite set of interdependent reactions that happen simultaneously. In an abstract sense, it can serve as a representation of any system whose evolution is driven by the transformation of its elements into different elements. The fundamental element of a chemical reaction is the \textbf{species}. The chemical species can encompass a range of entities, including chemical elements, molecules, or proteins. In the present context, the species represent various carbon pools involved in the system. A \textbf{complex} is a nonnegative linear combination of the species. Put another way, a complex is the set of species with associated nonnegative coefficients (called \textbf{stoichiometric coefficients}). A chemical reaction is typically written as  
$$
\text{Reactant complex} \rightarrow \text{Product complex},
$$
where the set of species on the left side of the equation (\textbf{reactant complex}) are consumed or transformed to form the set of species on the right side (\textbf{product complex}). We can view every complex in a CRN as a vector in a vector space called \textbf{species space}, whose coordinates refer to the coefficients or stoichiometry of the different species. In this way, every reaction may also be associated with a vector, called \textbf{reaction vector}. A \textbf{reaction vector} is formed by subtracting the reactant complex vector from the product complex vector. For example, the following network is a CRN with five species ($A_1$, $A_2$, $A_3$, $A_4$, and $A_5$) and seven reactions. 

\begin{equation}\label{eq: CRN}
\begin{blockarray}{cccc}
\text{Reaction} & \text{Reactant} & \text{Product} & \text{Reaction vector}  \\
A_1+2A_2 \rightarrow 2A_1+ A_2 & [1, 2, 0, 0, 0]^\top & [2, 1, 0, 0, 0]^\top & [1, -1, 0, 0, 0]^\top \\
2A_1+A_2 \rightarrow A_1+ 2A_2 & [2, 1, 0, 0, 0]^\top & [1, 2, 0, 0, 0]^\top & [-1, 1, 0, 0, 0]^\top \\
A_2 \rightarrow A_3 & [0, 1, 0, 0, 0]^\top & [0, 0, 1, 0, 0]^\top & [0, -1, 1, 0, 0]^\top \\
A_3 \rightarrow A_2 & [0, 0, 1, 0, 0]^\top & [0, 1, 0, 0, 0]^\top & [0, 1, -1, 0, 0]^\top \\
A_4 \rightarrow A_2 & [0, 0, 0, 1, 0]^\top & [0, 1, 0, 0, 0]^\top & [0, 1, 0, -1, 0]^\top \\
A_2 \rightarrow A_5 & [0, 1, 0, 0, 0]^\top & [0, 0, 0, 0, 1]^\top & [0, -1, 0, 0, 1]^\top \\
A_5 \rightarrow A_4 & [0, 0, 0, 0, 1]^\top & [0, 0, 0, 1, 0]^\top & [0, 0, 0, 1, -1]^\top 
\end{blockarray}
\end{equation}

Viewed as a directed graph, a CRN is said to be \textbf{weakly reversible} if the existence of a path from one complex $C_i$ to complex $C_j$ implies the existence of a path from $C_j$ to $C_i$. An example of this is the CRN in (\ref{eq: CRN}).

A group of complexes that are connected by arrows is referred to as a \textbf{linkage class}. The CRN above has two linkage classes: $\{A_1+2A_2 \rightleftarrows 2A_1+A2 \}$ and $\{ A_2 \rightleftarrows A_3, A_4 \rightarrow A_2, A_2 \rightarrow A_5, A_5 \rightarrow A_4 \}$.

The span or the set of all possible linear combinations of the reaction vectors is called the \textbf{stoichiometric subspace} of the network. The \textbf{rank} of a CRN refers to the dimension of the stoichiometric subspace (i.e., the maximum number of linearly independent reaction vectors). Therefore, the  stoichiometric subspace of the CRN in (\ref{eq: CRN}) has 4 basis vectors. That is, the rank of the CRN is 4.

A key concept in CRNT is the nonnegative structural index called \textbf{deficiency}. It is calculated as the number of complexes $n$ minus the number of linkage classes $\ell$ and the rank $s$. This index, independent of network size, measures the ``linear independence'' of reactions: higher deficiency indicates less linear independence. Even large or complex CRNs can have a deficiency of zero (\cite{SHFE2012}).

The CRN in (\ref{eq: CRN}) is a deficiency-zero CRN since it has 6 complexes, 2 linkage classes, and its rank is 4.

\subsubsection*{CRN representation  of the DAC model}

Given a box model of a global carbon cycle, a corresponding \textbf{CRN representation} can be set up using the procedure proposed by Arceo et al. \cite{AJMSM2015}. In this approach, one associates the reaction $A_i \rightarrow A_j$ to the carbon transfer from pool $A_i$ to pool $A_j$. Moreover, if the carbon transfer is influenced by some carbon pools (as indicated by the dashed arrows in the schematic diagram), say $\sum A_k$, all these species are added to both sides of $A_i \rightarrow A_j$ to form the chemical reaction 
\begin{equation}\label{eq:CRNrep}
\underbrace{A_i+ \left( \sum A_k \right)}_{\mbox{reactant complex}} \rightarrow \underbrace{A_j + \left( \sum A_k \right) }_{\mbox{product complex}}
\end{equation}
This process preserves the coordinates of the reaction vectors, which is important in describing the dynamics of the whole system (see Section \ref{sec:PL}). 

In the current model of the carbon cycle with DAC (see Fig. \ref{fig:DACPL}), the transfer of carbon from the atmosphere $A_2$ to land $A_1$ is influenced by $A_1$ and $A_2$. According to (\ref{eq:CRNrep}), the reaction associated with this process is  $A_2+ \{ A_1+A_2 \} \rightarrow A_1+\{A_1+A_2\}$ or simply $A_1+2A_2 \rightarrow 2A_1+A_2$. The carbon transfer from land to atmosphere is represented by the reaction $A_1 +A_2 \to 2A_2$ because the process is influenced by $A_2$.  This reaction can be translated (as described in \cite{JOHNSTON2014}), but without changing the stoichiometry, by adding $A_1$ to both sides of the reaction. The translation lowers the deficiency of the network from one (as done in \cite{DOA2018}) into zero. Hence, the atmosphere-land carbon transfer is depicted by the reversible reaction $A_1+2A_2 \rightleftarrows 2A_1+A_2$. 

In summary, the CRN representation of the DAC model is precisely the network specified in (\ref{eq: CRN}).

\subsubsection{Power-law kinetics of the DAC model}\label{sec:PL}

Given the CRN representation of the DAC model, we can associate the corresponding transfer power-law rates (provided in Figure \ref{fig:DACPL}) to each reaction:

\begin{equation}\label{eq:DACrates}
\begin{blockarray}{cc}
\text{Reaction} & \text{Reaction rates}  \\
R_1: A_1+2A_2 \rightarrow 2A_1+ A_2 & k_1A_1^{p_1}A_2^{q_1} \\
R_2: 2A_1+A_2 \rightarrow A_1+ 2A_2 & k_2A_1^{p_2}A_2^{q_2} \\
R_3:A_2 \rightarrow A_3 & a_m A_2 \\
R_4:A_3 \rightarrow A_2 & a_m\beta A_3 \\
R_5:A_4 \rightarrow A_2 & k_4 A_4 \\
R_6:A_2 \rightarrow A_5 & k_5 A_2 \\
R_7:A_5 \rightarrow A_4 & k_6 A_5
\end{blockarray}
\end{equation}

The behavior of a network is described by a system of nonlinear ODEs, which are derived from the CRN and the definition of reaction rate functions, or kinetics. This system of ODEs can be compactly represented using vector notation, encapsulating a collection of equations, each of which describes the evolution of a specific species within the CRN. Formally, the system can be expressed as:

\begin{equation}\label{eq:ode}
\dot{x}= \sum_{\mathscr{R}} \kappa_{y \rightarrow y'} (x) (y'-y)    
\end{equation}

Here, $x$ represents the vector of species concentrations, and the overdot signifies the derivative with respect to time. The scalar $\kappa_{y \rightarrow y'} (x)$ captures the rate at which the reaction $y \rightarrow y'$ proceeds. The term $(y'-y)$ corresponds to the reaction vector associated with the reaction $y \rightarrow y'$. The symbol $\mathscr{R}$ denotes the set of all reactions in the CRN, and its inclusion under the summation indicates that the sum is taken over all reactions in the network. A \textbf{positive steady state} of the system is a species composition $x$ with positive concentrations for which $\dot{x}=0$.

Equation (\ref{eq:ode}) highlights the role of the stoichiometric subspace in a CRN, as it constrains the system's dynamics. While species concentrations evolve over time, their trajectories are not free to roam throughout the species space. Instead, the concentrations are confined to translations of the stoichiometric subspace, termed as \textbf{stoichiometric compatibility classes}.

The CRN representation of the DAC model (in (\ref{eq:CRNrep})) must be endowed with \textbf{power-law kinetics} in order to reflect the ODE system in Equation (\ref{eq:DAC}). The power-law functions of a CRN representation are encoded using the \textbf{kinetic order matrix} $F$, where entry $F_{ij}$ encodes the kinetic order of the $j$th species in the $i$th reaction. Thus, for power law kinetic systems, the rate functions described in the ODE system in (\ref{eq:ode}) can be specified as

\begin{equation}\label{eq:PLK}
 \displaystyle K_{i}(x)=\kappa_i x^{(F_{i,*})^\top} \quad \text{for all } i \in \mathscr{R},
\end{equation}
where $F_{i,*}$ is the row vector containing the kinetic orders of the species of the reactant complex in the $i$th reaction. 

\begin{remark}
Equation (\ref{eq:PLK}) involves the operation where a vector is raised to a vector. In general, we define $x^y \in \mathbb{R}$ for $x,y\in \mathbb{R}^\mathscr{I}$ as
$$
x^y=\prod_{i \in \mathscr{I}}x_i^{y_i}.
$$
\end{remark}

\noindent Given the power-law rates in (\ref{eq:DACrates}), the kinetic order matrix of the current system is 

\begin{equation*}
F= \begin{blockarray}{cccccl}
A_1 & A_2 & A_3  & A_4 & A_5 &  \\
\begin{block}{[ccccc]l}
p_1 & q_1 & 0 & 0 & 0 & R_1 \\
p_2 & q_2 & 0 & 0 & 0 &  R_2 \\
0 & 1 & 0 &  0 & 0 & R_3 \\
0 & 0 & 1 &  0 & 0 & R_4 \\
0 & 0 & 0 &  1 & 0 & R_5 \\
0 & 1 & 0 &  0 & 0 & R_6 \\
0 & 0 & 0 &  0 & 1 & R_7 \\
\end{block}
\end{blockarray}.
\end{equation*}

Let $N$ be the matrix, called \textbf{stoichiometric matrix}, whose columns are the reaction vectors of the CRN. Then the system of ODEs specified in Equation (\ref{eq:ode}) can be rewritten as
\begin{equation}\label{eq:ODE2}
\dot{x}=NK(x),  
\end{equation}
where $K(x)$ is the vector which contains the reaction rates. It can be easily checked that if the stoichiometric matrix for the CRN representation of the DAC model and the vector the containing the reaction rates are specified, the ODE system in (\ref{eq:ODE2}) is precisely the system in (\ref{eq:DAC}), thereby verifying the dynamical equivalence of the power-law kinetic representation and the original system.

\begin{remark}
    Henceforth, we refer to the power-law kinetic representation of the model of carbon cycle system with DAC as the \textbf{DAC system}.
\end{remark}

\section{Steady-state analysis of the DAC system}

Since the power-law kinetic representation of the DAC system is weakly reversible and has zero deficiency, current theorems on power law kinetic systems on deficiency-zero networks \cite{TALABIS2017,MENDOZA2018} ensure the presence of a set of positive equilibria of the system. 

\begin{proposition}
The DAC system has a positive steady state. 
\end{proposition}

In fact, for a class of DAC system (later defined as positive and negative DAC systems), a parameterization of its steady state is provided in Proposition \ref{prop:steadystates}.

\subsection{Multistationarity analysis}

 A CRN is said to be \textbf{multistationary} or has the capability for multiple steady states if there is at least one stoichiometric compatibility class with at most two distinct positive steady states. Otherwise, the CRN is \textbf{monostationary}. 

It is found that the capacity of the DAC system to admit multiple steady states depends on values of the kinetic orders $p_1,p_2,q_1$, and $q_2$. More precisely, the multistationarity property is quickly decided based on the sign of the ratio $R_p$ or $R_q$ defined in Table \ref{tab1}. The discussion centers around these two values due to the structure of the so-called \textit{kinetic flux subspace} $\widetilde{S}$ of the system. Essentially, the kinetic flux subspace of a system is the kinetic analogue of the stoichiometric subspace. If the stoichiometric subspace is the span of the reaction vectors, the kinetic flux subspace is the span of the fluxes in terms of the kinetic vectors. Interestingly, a mathematical description of the set of positive steady states of a chemical kinetic system can be written as a vector element of the space that is perpendicular (i.e., orthogonal complement) to the system’s kinetic flux subspace: If the vector $x^*$ is any positive equilibrium of a system, the set of positive equilibria consists of vectors $x$ such that the vector $\log (x) - \log (x^*)$ resides in the orthogonal complement of kinetic flux subspace. (See Appendix \ref{appendix:PLP}.)

\begin{proposition}
For the DAC system, 
$$
(\widetilde{S})^\perp = \text{span }
\left\lbrace
\begin{bmatrix}
-1 \\
R_p \\
R_p \\
R_p \\
R_p
\end{bmatrix} \right\rbrace  
\text{where } R_p :=\dfrac{p_2-p_1}{q_2-q_1}, q_2 \neq q_1.
$$
Similarly,
$$
(\widetilde{S})^\perp = 
\text{span }
\left\lbrace
\begin{bmatrix}
-R_q \\
1 \\
1 \\
1 \\
1
\end{bmatrix}
\right\rbrace\quad \text{where } R_q :=\dfrac{q_2-q_1}{p_2-p_1}, p_2 \neq p_1.
$$
\end{proposition}
\begin{proof}
In a weakly reversible CRN,  $\widetilde{S}=\text{Im } \left( \widetilde{Y} I_a \right)$ (See Appendix for the definition of these matrices). For the DAC system,

$$
\widetilde{Y}=\kbordermatrix{
& A_1+2A_2 & 2A_1 + A_2 & A_2 & A_3 & A_4 & A_5\\
A_1 & p_1 & p_2 & 0 & 0 & 0 & 0 \\
A_2 & q_1 & q_2 & 1 & 0 & 0 & 0 \\
A_3 & 0 & 0 & 0 & 1 & 0 & 0 \\
A_4 & 0 & 0 & 0 & 0 & 1 & 0 \\
A_5 & 0 & 0 & 0 & 0 & 0 & 1 \\
} 
\text{and } 
I_a=\kbordermatrix{
& R_1 & R_2 & R_3 & R_4 & R_5 & R_6 & R_7 \\
A_1+2A_2 & -1 & 1 & 0 & 0 & 0 & 0 & 0 \\
2A_1+A_2 & 1 & -1 & 0 & 0 & 0 & 0 & 0 \\
A_2 & 0 & 0 & -1 & 1 & 1 & -1 & 0 \\
A_3 & 0 & 0 & 1 & -1 & 0 & 0 & 0 \\
A_4 & 0 & 0 & 0 & 0 & -1 & 0 & 1 \\
A_5 & 0 & 0 & 0 & 0 & 0 & 1 & -1 \\
}
$$
Hence,
$$
\widetilde{Y}  I_a =
\begin{bmatrix}
p_2-p_1 & p_1-p_2  & 0 & 0 & 0 & 0 & 0 \\
q_2-q_1 & q_1-q_2  & -1 & 1 & 1 & -1 & 0 \\
0 & 0  & 1 & -1 & 0 & 0 & 0 \\
0 & 0  & 0 & 0 & -1 & 0 & 1 \\
0 & 0  & 0 & 0 & 0 & 1 & -1 \\
\end{bmatrix}
$$
So,
$$
\widetilde{S}=\text{Im } \left( \widetilde{Y}  I_a \right) = \text{span }
\left\lbrace
\begin{bmatrix}
p_2-p_1 \\
q_2-q_1 \\
0 \\
0 \\
0
\end{bmatrix},
\begin{bmatrix}
0 \\
-1 \\
1 \\
0 \\
0
\end{bmatrix},
\begin{bmatrix}
0 \\
0 \\
1 \\
-1 \\
0
\end{bmatrix},
\begin{bmatrix}
0 \\
0 \\
0 \\
-1 \\
1
\end{bmatrix}
\right\rbrace.
$$
From here, it can be easily computed that $(\widetilde{S})^\perp = \text{span } 
\left\lbrace
\begin{bmatrix}
-1 & R_p & R_p & R_p & R_p
\end{bmatrix}^\top
\right\rbrace$  or $(\widetilde{S})^\perp =
\text{span } 
\left\lbrace
\begin{bmatrix}
-R_q & 1 & 1 & 1 & 1
\end{bmatrix}^\top
\right\rbrace
$.
\end{proof}

The ratio $R_p$ was crucial in the analysis of the power-law kinetic representation of the pre-industrial model of Anderies et al. done by \citet{FORMEN2023}. They identified and analyzed three distinct classes of Anderies systems, characterized by the sign of the ratio $R_p$. Here, the analysis of the DAC system will be largely characterized in a similar way. For convenience, we will call the different system classes of DAC as positive, negative, $P$-null and $Q$-null DAC systems:

\begin{definition}
We call the set of DAC systems with $R_p > 0$ (or $R_q>0$) as \textbf{positive DAC systems}. The set of DAC systems where $R_p < 0$ (or $R_q<0$) are termed \textbf{negative DAC systems}. DAC systems with $p_2-p_1=0$ and $q_2-q_1\neq 0$ are $\bm{P}$\textbf{-null DAC systems}. DAC systems with $q_2-q_1=0$ and $p_2-p_1\neq 0$ are $\bm{Q}$-\textbf{null DAC systems}. 
\end{definition}

\citet{MURE2012} provided a simple criterion to assess the uniqueness of a (complex balanced) steady state in a deficiency zero network. This is done by analyzing the sign vector connections between the stoichiometric subspace and the orthogonal complement of the kinetic flux subspace. 

\begin{theorem}(Proposition 3.2 of \cite{MURE2012}) \label{theorem:mure}
If for a weakly reversible generalized mass action system with $\sigma(\mathcal{S}) \cap \sigma(\widetilde{\mathcal{S}})^\perp \neq \{ 0 \}$, then there is a stoichiometric class with more than one complex balanced equilibrium.
\end{theorem}

Due to this, we have the following result:

\begin{proposition}
A positive DAC system is multistationary.
\end{proposition}

\begin{proof}
If $R_p>0$,   $\text{sign }(\widetilde{S})^\perp = \left\lbrace
\begin{bmatrix}
- \\
+ \\
+ \\
+ \\
+
\end{bmatrix} ,
\begin{bmatrix}
+ \\
- \\
- \\
- \\
-
\end{bmatrix} 
\right\rbrace
$. The stoichiometric subspace $S$ is spanned by the following vectors:
$$
\left\lbrace
\begin{bmatrix}
1 \\
-1 \\
0 \\
0 \\
0
\end{bmatrix},
\begin{bmatrix}
0 \\
-1 \\
1 \\
0 \\
0
\end{bmatrix},
\begin{bmatrix}
0 \\
1 \\
0 \\
-1 \\
0
\end{bmatrix},
\begin{bmatrix}
0 \\
-1 \\
0 \\
0 \\
1
\end{bmatrix}
\right\rbrace.
$$
Let $x \in S$ where
$$
x= a_1 \begin{bmatrix}
1 \\
-1 \\
0 \\
0 \\
0
\end{bmatrix}+
a_2 \begin{bmatrix}
0 \\
-1 \\
1 \\
0 \\
0
\end{bmatrix}+
a_3 \begin{bmatrix}
0 \\
1 \\
0 \\
-1 \\
0
\end{bmatrix}+
a_4 \begin{bmatrix}
0 \\
-1 \\
0 \\
0 \\
1
\end{bmatrix}=
\begin{bmatrix}
a_1 \\
-a_1-a_2+a_3-a_4 \\
a_2 \\
-a_3 \\
a_4
\end{bmatrix}.
$$
Choose $a_1>0, a_2<0,a_3>0,a_4<0,$ and $a_1+a_2>a_3-a_4$ so that $\text{sign} (x)= 
\begin{bmatrix}
+ \\
- \\
- \\
- \\
-
\end{bmatrix} 
$
and thus, $\text{sign} (S) \cap \text{sign} (\widetilde{S})^\perp \neq \{ 0 \}$. Therefore, any positive DAC system is multistationary. 
\end{proof}

For negative DAC system where $R_p<0$, we have  $\text{sign }(\widetilde{S})^\perp = \left\lbrace
\begin{bmatrix}
- \\
- \\
- \\
- \\
-
\end{bmatrix} ,
\begin{bmatrix}
+ \\
+ \\
+ \\
+ \\
+
\end{bmatrix} 
\right\rbrace
$. In order for $x\in S$ to have similar signs for all its components, say positive, necessarily
$$
a_1>0, a_2>0,a_3<0, a_4>0.
$$
However, the second component $-a_1-a_2+a_3-a_4<0$. Hence, it is not possible for a uniform positive sign for all components of $x$. Similarly, it is not possible to obtain a vector $x \in S$ with negative signs in all its components. Thus, $\text{sign }(S) \cap \text{sign }(\widetilde{S})^\perp = \{ 0 \}$ and Theorem \ref{theorem:mure} does not apply. For negative DAC systems, we turn to injectivity analysis.

A CRN with stoichiometric matrix $N$ is \textbf{injective} if for any distinct stoichiometrically compatible species vectors $x$ and $y$, we have $NK(x) \neq NK(y)$ for all kinetics $K$ endowed on the CRN. Note that if a CRN is injective, then it is monostationary. In other words, an injective CRN cannot support multiple positive steady states for any rate constants. Wiuf and Feliu \cite{WIUF2013, FELIU2013} established a criterion to identify if a system is injective. 

Let $M=N \cdot \text{diag}(z) \cdot F \; \text{diag}(k)$, where $N$ represents the stoichiometric matrix and $F$ is the kinetic order matrix of the PLK system. Consider the symbolic matrix $M^*$ created by using symbolic vectors $k=(k_1,\dots,k_m)$ and $z=(z_1,\dots,z_r)$. Assume $\{ \omega^1,\dots,\omega^d \}$ forms a basis of the left kernel of $N$, and $i_1,\dots,i_d$ represent row indices. Define the $m \times m$ matrix $M^*$ by substituting the $i_j$-th row of $M$ with $\omega^j$.

\begin{theorem}(\cite{WIUF2013,FELIU2013})\label{theorem:wiuffeliu}
The interaction network with power law kinetics and fixed kinetic orders is injective if and only if the determinant of $M^*$ is a non-zero homogeneous polynomial with all coefficients being positive or all being negative. 
\end{theorem}

We identify two subsets of negative DAC systems that are injective:

\begin{proposition}
The DAC system is injective, and hence monostationary, if any of the following cases hold:
\begin{enumerate}[(i)]
\item $p_1<0,p_2>0,q_1>0$, and $q_2<0$; or 
\item $p_1>0,p_2<0,q_1<0$, and $q_2>0$.
\end{enumerate}

\noindent For all other cases, the network is not injective.
\end{proposition}
\begin{proof}
Using the computational approach and Maple script provided by \citet{FELIU2013}, we obtain the determinant of $M^*$ for the DAC system: 

\begin{align*}
det =&-p_1k_1k_2k_3k_4z_1z_4z_5z_6 – p_1k_1k_2k_3k_5z_1z_4z_6z_7 -p_1k_1k_2k_4k_5z_1z_3z_5z_7 \\
&-p_1k_1k_3k_4k_5z_1z_4z_5z_7 +p_2k_1k_2k_3k_4z_2z_4z_5z_6 + p_2k_1k_2k_3k_5z_2z_4z_6z_7 \\
&+p_2k_1k_2k_4k_5z_2z_3z_5z_7 +p_2k_1k_3k_4k_5z_2z_4z_5z_7 + q_1k_2k_3k_4k_5z_1z_4z_5z_7 \\
&-q_2k_2k_3k_4k_5z_2z_4z_5z_7
\end{align*}

\noindent Hence, for $p_1<0,p_2>0,q_1>0$, and $q_2<0$, all the terms are positive, and for $p_1>0,p_2<0,q_1<0$, and $q_2>0$, all the terms are negative. In both cases, the networks are injective by Theorem \ref{theorem:wiuffeliu} and hence, monostationary. In all other cases, the systems are non-injective, which is a necessary condition for multistationarity.
\end{proof}

Finally, for $P$-null and $Q$-null DAC systems, the analysis is based on a specific type of network decomposition known as independent decomposition. A network decomposition is said to be \textbf{independent} if the stoichiometric subspace of the CRN is the direct sum of the stoichiometric subspaces of the subnetworks. \citet{FEIN1987} demonstrated that, in an independent decomposition, the intersection of the set of positive steady states of the subnetworks coincides with the set of positive steady states of the entire network (see  Appendix \ref{appendix:IndependentDecomp}). By examining an independent decomposition of the $P$-null and $Q$-null DAC systems, we observe that:

\begin{proposition}
Any $P$-null or $Q$-null DAC system is monostationary. 
\end{proposition}
\begin{proof}
Consider the (finest) independent decomposition of DAC that contains two subnetworks:

\begin{align*}
    \mathscr{N}_1 &= \{A_1+2A_2 \rightleftarrows 2A_1+A_2, A_2 \rightleftarrows A_3  \}, \\
    \mathscr{N}_2 &= \{ A_4 \rightarrow A_2, A_2 \rightarrow A_5, A_5 \rightarrow A_4 \}
\end{align*}

\noindent The subnetwork $\mathscr{N}_1$ is identical to the power-law kinetic representation of the pre-industrial system of the Anderies system studied by \citet{FORMEN2023}. They showed that the null Anderies system cannot exhibit multiple steady-states or monostationary. The other subnetwork ($\mathscr{N}_2$) is a mass action system that is weakly reversible and has zero deficiency. By the classic Deficiency Zero Theorem for mass action systems \cite{FEIN1972,HORNJACK1972,HORN1972}, $\mathscr{N}_2$ is monostationary. Since both subsystems are monostationary, then the whole system is also monostationary (s. Theorem \ref{feinberg theorem}). 
\end{proof}

\subsection{Absolute concentration robustness}

\textbf{Absolute concentration robustness} or ACR refers to a condition in which the concentration of a species in a network attains the same value in every positive steady-state set by parameters and does not depend on initial conditions \cite{SHFE2010}. This implies that if an important variable, such as $A_2$ (representing atmospheric CO$_2$ concentration), exhibits ACR, its stability is guaranteed even when other variables fluctuate. Notably, conducting an ACR analysis for the DAC system is straightforward by examining the values of $R_p$ and $R_q$ and applying the \textit{Species Hyperplane Criterion} introduced by \citet{LLMM2022} (detailed in Appendix \ref{appendix:PLP}).

\begin{proposition}
A DAC system that is a
\begin{enumerate}[(i)]
    \item positive or negative system has no ACR species; 
    \item $P$-null system has ACR species consisting precisely of $A_2$, $A_3$, $A_4$ and $A_5$; and
    \item $Q$-null system has ACR in $A_1$.
\end{enumerate}    
\end{proposition}

\begin{proof}
The \textit{Species Hyperplane Criterion} (stated in Appendix \ref{appendix:PLP}) guarantees that a system has ACR species if and only if the vector coordinates corresponding to these species are zero for all basis vectors in space $(\widetilde{S})^\perp$.  As noted earlier, for the DAC system, $(\widetilde{S})^\perp = \text{span } 
\left\lbrace
\begin{bmatrix}
-1 & R_p & R_p & R_p & R_p
\end{bmatrix}^\top
\right\rbrace$  or $(\widetilde{S})^\perp =
\text{span } 
\left\lbrace
\begin{bmatrix}
-R_q & 1 & 1 & 1 & 1
\end{bmatrix}^\top
\right\rbrace
$.
Hence, a DAC system with positive or negative $R_p$ or $R_q$ has no ACR species. The DAC system with $R_p=0$ has ACR species consisting precisely of $A_2$, $A_3$, $A_4$ and $A_5$. If $R_q=0$, the system has ACR in $A_1$.
\end{proof}

This means that if we desire that $A_2$ or the CO$_2$ concentration in the atmosphere be stable irrespective of the initial conditions, we would like $R_p$ to be equal to zero. To achieve this, $p_1$ (the kinetic order of land photosynthesis interaction) must be equal to $p_2$ (the kinetic order of land respiration interaction) but $q_1$ (the kinetic order of atmosphere photosynthesis interaction) must not be equal to $q_2$ (the kinetic order of atmosphere respiration interaction).

\subsection{Conditions for atmospheric carbon reduction}

We provide different conditions under which a set of initial conditions, along with all positive steady states within the corresponding stoichiometric classes, leads to a genuine reduction in the atmospheric carbon pool. That is, for a set of initial conditions $\mathcal{A}^0=\{A_1^0, A_2^0, A_3^0, A_4^0, A_5^0\}$ and steady state $\mathcal{A}^*=\{A_1^*, A_2^*, A_3^*, A_4^*, A_5^*\}$ in the corresponding stoichiometric class, we have $A_2^*<A_2^0$.

The approach taken here is to use total amounts or conserved quantities in a kinetic system. Generally, for any element $w \in S^\perp$ and $x \in \Omega$, 
$$
w f(x) = w \frac{dx}{dt} = 0, 
$$
since $\text{Im}\, f \subseteq S$. If $w = (w_1, \dots, w_m)$ and $x = (x_1, \dots, x_m)$, then 
$$
0 = \sum w_i \frac{dx_i}{dt} = \frac{d}{dt} \left( \sum w_i x_i \right), 
$$
which implies that $T := \sum w_i x_i$ is a constant—this is called a \textbf{conserved quantity}. Clearly, any two elements in a stoichiometric compatibility class have the same conserved quantity or total amount.

The DAC system is a conservative CRN (in the sense of Definition \ref{def:conservative} in Appendix \ref{appendix: CRNT}) since $(1, 1, 1, 1, 1)$ is a basis for $S^\perp$. This implies that for a set of initial conditions $\mathcal{A}^0$ and steady state $\mathcal{A}^*$ in the corresponding stoichiometric class,
$$
A^0_1 + \dots + A^0_5 = A^*_1 + \dots + A^*_5.
$$

\subsubsection{A necessary condition for $A_2$ reduction in positive/negative DAC systems}

Observe that the set of positive equilibria of the positive or negative CDAC system can be described as follows:
\begin{proposition}\label{prop:steadystates}
The set of positive equilibria of the positive or negative CDAC system can be parametrized by $A_2$ as follows

$$
 A_1=\left(\dfrac{k_2}{k_1} A_2^{q_2 - q_1} \right)^\frac{1}{p_1-p_2}, \quad p_1 \neq p_2 
$$

$$
    A_2 = A_2, \quad\quad
    A_3 = \frac{1}{\beta} A_2, \quad\quad
    A_4 = \frac{k_5}{k_4} A_2 \quad\quad
    A_5 = \frac{k_5}{k_6} A_2.
$$
\end{proposition}

\begin{proof}
Set the equations in (\ref{eq:DAC}) to 0 and solve for $A_2$. For instance, setting the first equation in the system to 0, we get
$$
A_1^{p_1 - p_2} = \dfrac{k_2}{k_1} A_2^{q_2 - q_1}
$$
If $p_1 \neq p_2$, 
$A_1=\left(\dfrac{k_2}{k_1} A_2^{q_2 - q_1} \right)^\frac{1}{p_1-p_2}$.
\end{proof}

\begin{remark}
    If we let $P:=p_2 - p_1$,  the equilibrium value of $A_1$ can be written as 
    $$
    A_1 = \left( \dfrac{k_1}{k_2} \right)^{\sfrac{1}{P}} A_2^{-R_q}
    $$
\end{remark}

\noindent Denote
\begin{align*}
    \mathsf{SUM}^0 &:= A_1^0 +A_3^0 + A_4^0 +A_5^0 ; \text{ and} \\
    \mathsf{SUM}^* &:= A_1^* +A_3^* + A_4^* +A_5^*. \\
\end{align*}
A set of initial conditions $\mathcal{A}^0$ determines a unique stoichiometric class $\mathcal{S}^0$. Suppose there is another point in $\mathcal{S}^0$ with $A_2=\lambda A_2^0$, where $0<\lambda<1$. Moreover, suppose there is a positive or negative DAC system with an equilibrium in $\mathcal{S}^0$ whose $A_2$-value is $\lambda A_2^0$. For all positive equilibria in $\mathcal{S}^0$, we have
$$
\mathsf{SUM}^0 + A_2^0 = \mathsf{SUM}^* + A_2^* .
$$

\noindent From the set of positive equilibria of positive and negative CDAC in Proposition \ref{prop:steadystates}, we have 

\begin{equation}\label{eq:necessary}
\mathsf{SUM}^0 + A_2^0 = \left( \dfrac{k_1}{k_2} \right)^{\sfrac{1}{P}} \left(\lambda A_2^0 \right) ^{-R_q} + \left(\frac{1}{\beta} + \frac{k_5}{k_4} + \frac{k_5}{k_6} + 1\right) \lambda A_2^0 
\end{equation}

\noindent From the last equation, we state the following necessary condition for the reduction of $A_2$ from its initial value to its steady-state value:

\begin{proposition}
All the values of $P$, $R_q$, $\beta$, and the rate constants satisfying Equation (\ref{eq:necessary}) enable atmospheric carbon reduction from the initial $A_2^0$ to the steady-state value $\lambda A_2^0$. 
\end{proposition}

\subsubsection{Sufficient conditions for atmosphere carbon reduction}


Since the DAC system is conservative, each stoichiometric class is compact (s. Appendix 1 of \citet{HORNJACK1972}). Hence, the continuous maps $\mathsf{pr}_i: \mathbb{R}^\mathscr{S}_{>0} \to \mathbb{R}_>$ and their sums attain maxima and minima on any stoichiometric class.

\begin{proposition}
    Consider the $P$-null DAC system with set of initial conditions $\mathcal{A}^0$ and positive steady state $\mathcal{A}^*$ in the corresponding $\mathcal{S}^0$ (viewed as a compact subset of $\mathbb{R}^\mathscr{S}_{>0}$). Then if
    $$
    \left(\dfrac{k_1}{k_2}\right)^{\frac{1}{q_2-q_1}} <T-M'', 
    $$
    where $T$ is the conserved quantity and $M''$ is the maximum of $\mathsf{pr}_1+\mathsf{pr}_3+\mathsf{pr}_4+\mathsf{pr}_5$ on $\mathcal{S}^0$, then $  A_2^*<A_2^0$. 
\end{proposition}
\begin{proof}
From the first equation of the ODE system in (\ref{eq:DAC}), we observe that for a $P$-null DAC system (i.e., $p_1=p_2$),  $A_2^*=\left(\dfrac{k_1}{k_2}\right)^{\frac{1}{q_2-q_1}}$. Moreover, $$T-M'' \leq T- \left( A_1^0 + A_3^0 + A_4^0 + A_5 ^0\right)=A_2^0.$$
\end{proof}


\begin{proposition}
    Consider the positive or negative DAC system with set of initial conditions $\mathcal{A}^0$ and positive steady state $\mathcal{A}^*$ in the corresponding $\mathcal{S}^0$. Let $m'$ be the minimum of the continuous map $\mathsf{pr}_2$ on $\mathcal{S}^0$. Then if
    $$
    1+\dfrac{M''}{m'}<\left(\dfrac{k_2}{k_1}\right)^\frac{1}{p_1-p_2}(m')^\frac{q_2-q_1}{p_1-p_2} + \frac{1}{\beta} + \frac{k_5}{k_4} + \frac{k_5}{k_6}, 
    $$
   where $M''$ is the maximum of $\mathsf{pr}_1+\mathsf{pr}_3+\mathsf{pr}_4+\mathsf{pr}_5$ on $\mathcal{S}^0$, then $ A_2^*<A_2^0$. 
\end{proposition}

\begin{proof}
Denote $A_1^*+ A_2^*+ A_3^* + A_4^* + A_5^*$ by $=A_2^* \left( \mathsf{SUM}^+ \right)$,
where in $\mathsf{SUM}^+$, the summands are given in the parametrization of the equilibria set of positive/negative DAC system (Prop. \ref{prop:steadystates}). Now,
\begin{align*}
    1 < \dfrac{A^0_2}{A_2^*} & \Longleftrightarrow  1+ \dfrac{(\mathsf{pr}_1+\mathsf{pr}_3+\mathsf{pr}_4+\mathsf{pr}_5)(\mathcal{A}^0)}{A_2^*} \\
     & < \dfrac{A^0_2}{A_2^*} + \dfrac{(\mathsf{pr}_1+\mathsf{pr}_3+\mathsf{pr}_4+\mathsf{pr}_5)(\mathcal{A}^0)}{A_2^*}=\mathsf{SUM}^+ .
\end{align*}
\noindent We have 
$$
1+ \dfrac{(\mathsf{pr}_1+\mathsf{pr}_3+\mathsf{pr}_4+\mathsf{pr}_5)(\mathcal{A}^0)}{A_2^*}  \leq 1+ \dfrac{M''}{m'}
$$
\noindent by the definition of numerator and denominator. Furthermore, $$ \left(\dfrac{k_2}{k_1}\right)^\frac{1}{p_1-p_2}(m')^\frac{q_2-q_1}{p_1-p_2} + \frac{1}{\beta} + \frac{k_5}{k_4} + \frac{k_5}{k_6}<\mathsf{SUM}^+.$$ Hence, the RHS of the first equivalence above is fulfilled and $A_2^*<A_2^0$.
\end{proof}

\section{Conclusion, Summary, and Outlook}

We analyzed a global carbon cycle system that incorporates direct air capture technology by utilizing the tools and insights in Chemical Reaction Network Theory. The innovative aspect of this approach lies in its ability to promptly offer crucial insights into the system's long-term dynamics by focusing on the network's topological structure and kinetics, eliminating the necessity to specify system parameters.

By examining a dynamically equivalent reaction network of a global carbon cycle system with DAC technology, this study efficiently identified three crucial dynamic features: the existence of positive steady states, the possibility of multiple steady states, and the absolutely robust concentration levels of carbon pools. Irrespective of kinetic orders and rate constants, the DAC system is expected to exhibit a positive steady state. Additionally, assessments concerning the system's multistationarity and ACR traits are based on the sign of ratios $R_p$ and $R_q$. Table \ref{table:summary} outlines the results discussed earlier, connecting the signs of $R_p$ and $R_q$ to the dynamic characteristics of the associated DAC system.

\begin{table}
\begin{center}
\begin{minipage}{\textwidth}
\begin{tabular}{|c|l|} 
\hline
\textbf{Property} & \textbf{DAC system}  \\ 
\hline
\multirow{1}{*}{Existence of at least one steady state}	&	True for all systems	\\
\hline
\multirow{4}{*}{Capacity for multiple steady states}	&	$R_p = 0$: only one steady state	\\
\hhline{~-}   &  $R_q = 0$: only one steady state  \\
\hhline{~-}  &  $R_p$ or $R_q > 0$: all parameter combinations	 \\
  &  result in more than one steady state 	 \\
\hhline{~-}  &  $R_p$ or $R_q < 0$: some parameter combinations  	 \\
 &  may result in more than one steady state  	 \\
\hline
\multirow{4}{*}{ACR}	&	$R_p = 0$: ACR in $A_2, A_3, A_4, A_5$	\\
\hhline{~-}   &  $R_q = 0$: ACR in $A_1$ \\
\hhline{~-}  &  $R_p$ or $R_q > 0$: no ACR in any species 	 \\
\hhline{~-}  &  $R_p$ or $R_q < 0$: no ACR in any species 	 \\
\hline
\end{tabular}
\end{minipage}
\end{center}
\caption{Summary of the dynamic properties of the DAC system. } \label{table:summary}
\end{table} 

It can be seen from Table \ref{table:summary}, that the desirable outcome of a unique and stable concentration of carbon dioxide in the atmosphere may be realized if $R_p = 0$. No tipping points from the existence of multistationarity would be expected. 

The approach utilized in this research could prove valuable for rapidly evaluating other negative emission technologies (NETs). It can efficiently determine if the system fails to meet specified crucial criteria (such as the absence of a positive steady state in the long run or bistability), prompting a reassessment of the technology's deployment.

Although the system examined is currently limited in scope, there is potential to refine and transfer the suggested framework to more complex carbon cycle models. When dealing with a broader or CRN representation of a carbon cycle, incorporating network decomposition theory in CRNT to dissect the system into smaller elements could be a promising approach.

The idea of ``planetary boundaries," highlighted by \citet{AND2013}, has had a profound influence on the global sustainability community, as demonstrated in the research conducted by \citet{Steffen2015}. Our ongoing research efforts focus on developing kinetic representations for various CDR methods such as bioenergy with carbon capture and storage and ocean fertilization. \citet{TAN2022} have stressed the significance of optimizing combinations or ``portfolios" of NETs. To address this challenge, we aim to investigate other combinations of NETs to determine if these may exhibit steady-state multiplicity.



\bibliography{rncdr}

\appendix

\section{Fundamentals of reaction networks and kinetic systems }\label{appendix: CRNT}

As supplementary material, we provide a formal presentation of the relevant concepts and results related to chemical reaction networks and chemical kinetic systems.

\subsubsection*{Notation}
We denote the real numbers by $\mathbb{R}$, the non-negative real numbers by $\mathbb{R}_{\geq0}$ and the positive real numbers by $\mathbb{R}_{>0}$.  Objects in reaction systems are viewed as members of vector spaces. Suppose $\mathscr{I}$ is a finite index set. By $\mathbb{R}^\mathscr{I}$, we mean the usual vector space of real-valued functions with domain $\mathscr{I}$.  If $x \in \mathbb{R}_{>0}^\mathscr{I}$ and $y \in \mathbb{R}^\mathscr{I}$, we define $x^y \in \mathbb{R}_{>0}$ by
$
x^y= \prod_{i \in \mathscr{I}} x_i^{y_i} .
$
Let $x \wedge y$ be the component-wise minimum, $(x \wedge y)_i = \min (x_i, y_i)$.
The vector $\log x\in \mathbb{R}^\mathscr{I}$,where $x \in \mathbb{R}_{>0}^\mathscr{I}$, is given by 
$(\log x)_i = \log x_i,  \text{ for all } i \in \mathscr{I}.$  The \textit{support} of $x \in \mathbb{R}^\mathscr{I}$, denoted by $\text{supp } x$, is given by
$ \text{supp } x := \{ i \in \mathscr{I} \mid x_i \neq 0 \}.$

\subsection{Fundamentals of chemical reaction networks}

We begin with the formal definition of a chemical reaction network or CRN. 

\begin{definition}
A \textbf{chemical reaction network} or CRN is a triple $\mathscr{N}:= (\mathscr{S,C,R})$ of nonempty finite sets $\mathscr{S}$, $\mathscr{C}$, and $\mathscr{R}$, of $m$ \textbf{species}, $n$ \textbf{complexes}, and $r$ \textbf{reactions}, respectively, where $\mathscr{C} \subseteq \mathbb{R}_{\geq 0}^\mathscr{S}$ and $\mathscr{R} \subset \mathscr{C} \times \mathscr{C}$ satisfying the following properties:
\begin{enumerate}[(i)]
    \item $(y,y) \notin \mathscr{R}$ for any $y \in \mathscr{C}$;
    \item for each $y \in \mathscr{C}$, there exists $y' \in \mathscr{C}$ such that $(y,y')\in \mathscr{R}$ or $(y',y)\in \mathscr{R}$.
\end{enumerate}
\end{definition}
\noindent For $y \in \mathscr{C}$, the vector $$y=\displaystyle{\sum_{A \in \mathscr{S}}} y_A A,$$ where $y_A$ is the \textbf{stoichiometric coefficient} of the species $A$. In lieu of $(y,y')\in \mathscr{R}$, we write the more suggestive notation  $y \rightarrow y'$. In this reaction, the vector $y$ is called the \textbf{reactant complex} and $y'$ is called the \textbf{product complex}. 

CRNs can be viewed as directed graphs where the complexes are vertices and the reactions are arcs. The (strongly) connected components are precisely the \textbf{(strong) linkage classes} of the CRN. A strong linkage class is a \textbf{terminal strong linkage class} if there is no reaction from a complex in the strong linkage class to a complex outside the given strong linkage class. 

\begin{definition}
A CRN with $n$ complexes, $n_r$ reactant complexes, $\ell$ linkage classes, $s\ell$ strong linkage classes, and $t$ terminal strong linkage classes is 
\begin{enumerate}[(i)]
    \item \textbf{weakly reversible} if $s\ell = \ell$;
    \item \textbf{t-minimal} if $t=\ell$;
    \item \textbf{point terminal} if $t=n-n_r$; and
    \item \textbf{cycle terminal} if $n-n_r=0$.
\end{enumerate}
\end{definition}

For every reaction, we associate a \textbf{reaction vector}, which is obtained by subtracting the reactant complex from the product complex. From a dynamic perspective, every reaction $ y \rightarrow y' \in \mathscr{R}$ leads to a change in species concentrations proportional to the  reaction vector $ \left( y' – y \right) \in \mathbb{R}^\mathscr{S}$. The overall change induced by all the reactions lies in a subspace of $\mathbb{R}^\mathscr{S}$ such that any trajectory in $\mathbb{R}^\mathscr{S}_{>0}$ lies in a coset of this subspace. 

\begin{definition}
The \textbf{stoichiometric subspace} of a network $\mathscr{N}$ is given by
$$ \mathcal{S} := \text{span } \{ y' – y \in \mathbb{R}^\mathscr{S} \mid y \rightarrow y' \in \mathscr{R} \}.$$
The \textbf{rank} of the network is defined as $s:= \dim \mathcal{S}$. For $x \in \mathbb{R}^\mathscr{S}_{>0}$, its \textbf{stoichiometric compatibility class} is defined as $(x+\mathcal{S}) \cap \mathbb{R}^\mathscr{S}_{ \geq 0}$. Two vectors $x^{*}, x^{**} \in  \mathbb{R}^\mathscr{S}$ are \textbf{stoichiometrically compatible} if $ x^{**}-x^{*} \in \mathcal{S}$.
\end{definition}

\begin{definition}
A CRN with stoichiometric subspace $S$ is said to be \textbf{conservative} if there exists a positive vector $x \in \mathbb{R}^\mathscr{S}_>$ such that $S^\perp \cap \mathbb{R}^\mathscr{S}_> \neq \emptyset$. 
\end{definition}

An important structural index of a CRN, called \textit{deficiency}, provides one way to classify networks.

\begin{definition}
The \textbf{deficiency} $\delta$ of a CRN with $n$ complexes, $\ell$ linkage classes, and rank $s$ is defined as $\delta:=n-\ell-s$.
\end{definition}

\subsection{Fundamentals of chemical kinetic systems}
It is generally assumed that the rate of a reaction $y \rightarrow y' \in \mathscr{R}$ depends on the concentrations of the species in the reaction. The exact form of the non-negative real-valued rate function $K_{ y \rightarrow y'}$ depends on the underlying \textit{kinetics}.

The following definition of kinetics is expressed in a more general context than what one typically finds in CRNT literature.
\begin{definition}
A \textbf{kinetics} for a network $\mathscr{N}=(\mathscr{S,C,R})$ is an assignment to each reaction $y \rightarrow y' \in \mathscr{R}$ a rate function $ K_{ y \rightarrow y'}: \Omega_K \rightarrow \mathbb{R}_{\geq 0}$, where $\Omega_K$ is a set such that $\mathbb{R}^\mathscr{S}_{> 0} \subseteq \Omega_K \subseteq \mathbb{R}^\mathscr{S}_{\geq 0}$, $x  \wedge  x^{*} \in \Omega_K$ whenever $x, x^{*} \in \Omega_K$, and $ K_{ y \rightarrow y'} (x) \geq 0$ for all $x \in \Omega_K$. A kinetics for a network $\mathscr{N}$ is denoted by $K:\Omega_K \rightarrow \mathbb{R}^\mathscr{R}_{\geq 0}$ (\cite{WIUF2013}). A \textbf{chemical kinetics} is a kinetics $K$ satisfying the condition that for each $y \rightarrow y' \in \mathscr{R}$, $ K_{ y \rightarrow y'} (x) >0$ if and only if $\text{supp } y \subset \text{supp } x$. The pair $(\mathscr{N},K)$ is called a \textbf{chemical kinetic system} (\cite{AJLM2017}). 
\end{definition}

The system of ordinary differential equations that govern the dynamics of a CRN is defined as follows.

\begin{definition}\label{def:ODE}
The \textbf{ordinary differential equation (ODE)} associated with a chemical kinetic system $(\mathscr{N},K)$ is defined as 
$ \dfrac{dx}{dt}=f(x)$ with \textbf{species formation rate function} 
\begin{equation}\label{eq:sfrf}
    f(x)= \sum_{ y \rightarrow y' \in \mathscr{R}} K_{ y \rightarrow y'} (x) (y'-y).
\end{equation}
A \textbf{positive equilibrium} or \textbf{steady state} $x$ is an element of $\mathbb{R}^\mathscr{S}_{>0}$ for which $f(x) = 0$.
\end{definition}

\begin{definition}\label{def:equilibria}
The \textbf{set of positive equilibria} or \textbf{steady states} of a chemical kinetic system $(\mathscr{N},K)$ is given by 
$$ E_+ (\mathscr{N},K) = \{ x \in \mathbb{R}^\mathscr{S}_{>0} \mid f(x) = 0 \}. $$
For brevity, we also denote this set by $E_+$. The chemical kinetic system is said to be \textbf{multistationary} (or has the capacity to admit \textbf{multiple steady states}) if there exist positive rate constants such that $\mid E_+ \cap \mathcal{P}\mid \geq 2$ for some positive stoichiometric compatibility class $\mathcal{P}$. On the other hand, it is \textbf{monostationary} if $\mid E_+ \cap \mathcal{P}\mid \leq 1$ for all positive stoichiometric compatibility class $\mathcal{P}$.
\end{definition}

\begin{definition}
The reaction vectors of a CRN $ (\mathscr{S,C,R})$ are \textbf{positively dependent} if for each reaction $y \rightarrow y' \in \mathscr{R}$, there exists a positive number $k_{ y \rightarrow y'}$ such that $\sum_{y \rightarrow y' \in \mathscr{R}}k_{ y \rightarrow y'} (y'-y)=0$. 
\end{definition}

\begin{remark}
 In view of Definitions \ref{def:ODE} and  \ref{def:equilibria}, a necessary condition for a chemical kinetic system to admit a positive steady state is that its reaction vectors are positively dependent.  
\end{remark}

\begin{definition}\label{def:conservative}
A CRN is said to be \textbf{conservative} if there is a vector $v \in S^\perp$ such that $v \in \mathbb{R}^\mathscr{S}_{>0}$, or equivalently, $S^\perp \cap \mathbb{R}^\mathscr{S}_{>0} \neq \emptyset$.
\end{definition}

To reformulate the species formation rate function in Eq. (\ref{eq:sfrf}), we consider the natural basis vectors $\omega_i \in \mathbb{R}^\mathscr{I}$ where $i \in \mathscr{I}=\mathscr{C}$ or $\mathscr{R}$ and define 
\begin{enumerate}
    \item[(i)] the \textbf{molecularity map} $Y: \mathbb{R}^\mathscr{C} \rightarrow \mathbb{R}^\mathscr{S}$ with $Y(\omega_y)=y$;
    \item[(ii)] the \textbf{incidence map} $I_a: \mathbb{R}^\mathscr{R} \rightarrow \mathbb{R}^\mathscr{S}$ with $I_a (\omega_{y \rightarrow y'})= \omega_{y'} - \omega_y$; and
    \item[(iii)] the \textbf{stoichiometric map} $N: \mathbb{R}^\mathscr{R} \rightarrow \mathbb{R}^\mathscr{S}$ with $N= YI_a$.
\end{enumerate}
Hence, Eq. (\ref{eq:sfrf}) can be rewritten as $f(x)=YI_aK(x)=NK(x).$ The positive steady states of a chemical kinetic system that satisfies $I_a K(x)=0$ are called \textit{complex balancing equlibria}. 
\begin{definition}
The \textbf{set of complex balanced equilibria} of a chemical kinetic system $(\mathscr{N},K)$ is the set
$$ Z_+(\mathscr{N},K) = \{ x\in \mathbb{R}^\mathscr{S}_{>0} \mid I_a K(x) =0 \} \subseteq E_+(\mathscr{N},K).$$
A chemical kinetic system is said to be \textbf{complex balanced} if it has a complex balanced equilibrium. 
\end{definition}

We define power law kinetics through the $r \times m$ \textbf{kinetic order matrix} $F=[F_{ij}]$, where $F_{ij} \in \mathbb{R}$ encodes the kinetic order the $j$th species of the reactant complex in the $i$th reaction. Further, consider the \textbf{rate vector} $k \in \mathbb{R}^\mathscr{R}_{>0}$, where $k_i \in \mathbb{R}_{>0}$ is the rate constant in the $i$th reaction. 

\begin{definition}\label{def:PLK}
A kinetics $K: \mathbb{R}^\mathscr{S}_{>0} \rightarrow \mathbb{R}^\mathscr{R}$ is a \textbf{power law kinetics} or \textbf{PLK} if
$$\displaystyle K_{i}(x)=k_i x^{F_{i,*}} \quad \text{for all } i \in \mathscr{R},$$
where $F_{i,*}$ is the row vector containing the kinetic orders of the species of the reactant complex in the $i$th reaction.
\end{definition}
 

\begin{definition} A PLK system has \textbf{reactant-determined kinetics} (or of type \textbf{PL-RDK}) if for any two \textbf{branching reactions} $i$, $j \in \mathscr{R}$ (i.e., reactions sharing a common reactant complex), the corresponding rows of kinetic orders in $F$ are identical. That is, $F_{ih}=F_{jh}$ for all $h  \in \mathscr{S}$.  
\end{definition}



\begin{definition}\label{def:Ytilde}
The $m \times n$ matrix $\widetilde{Y}$ defined by \citet{MURE2012} is the matrix  $( \widetilde{Y})_{ij} = F_{ki}$ if $j$ is a reactant complex of reaction $k$ and $( \widetilde{Y})_{ij} = 0$, otherwise. The $m \times n_r$ $\bm{T}$\textbf{-matrix} is the truncated $\widetilde{Y}$ where the non-reactant columns are deleted and $n_r$ is the number of reactant complexes. 
\end{definition}



\subsection{Independent decomposition of a CRN} \label{appendix:IndependentDecomp}
Decomposition theory was initiated by M. Feinberg in his 1987 review paper \cite{FEIN1987}. He introduced the general concept of a network decomposition of a CRN as a union of subnetworks whose reaction sets form a partition of the network’s set of reactions. He also introduced the so-called \textit{independent decomposition} of chemical reaction networks.

\begin{definition}
A decomposition of a CRN $\mathscr{N}$ into $k$ subnetworks of the form $\mathscr{N}=\mathscr{N}_1 \cup \cdots \cup \mathscr{N}_k$ is \textbf{independent} if its stoichiometric subspace is equal to the direct sum of the stoichiometric subspaces of its subnetworks, i.e., $\mathcal{S}=\mathcal{S}_1 \oplus \cdots \oplus \mathcal{S}_k$.
\end{definition}

For an independent decomposition, Feinberg concluded that any positive equilibrium of the “parent network” is also a positive equilibrium of each subnetwork.

\begin{theorem}[Rem. 5.4, \cite{FEIN1987}]  \label{feinberg theorem}
Let $(\mathscr{N},K)$ be a chemical kinetic system with partition $\{\mathscr{R}_1, \dots, \mathscr{R}_k \}$. If $\mathscr{N}=\mathscr{N}_1 \cup \cdots \cup\mathscr{N}_k$ is the network decomposition generated by the partition  and $E_+(\mathscr{N}_i,K_i)= \{ x \in \mathbb{R}^\mathscr{S}_{>0} \mid N_i K_i(x) = 0, i=1,\dots,k \}$, then 
$ \bigcap_{i=1}^kE_+ (\mathscr{N}_i, K_i)\subseteq E_+ (\mathscr{N}, K)$. If the network decomposition is independent, then equality holds.
\end{theorem}

\subsection{Absolute concentration robustness in PLP systems}\label{appendix:PLP}

Lao et al. \cite{LLMM2022} introduced the concept of positive equilibria log-parametrized (PLP) kinetic system:

\begin{definition}
For a reaction network $\mathscr{N}$ with species $\mathscr{S}$, a \textbf{log-parametrized (LP) set} is a non-empty set of the form $$E(P,x^*)= \{ x\in \mathbb{R}^\mathscr{S}_{>0} \mid \log x -\log x^* \in P^\perp \},$$ where $P$ (called the LP set's \textbf{flux subspace}) is a subspace of $\mathbb{R}^\mathscr{S}$, $x^*$ (called the LP's \textbf{reference point}) is a given element of $\mathbb{R}^\mathscr{S}_{>0}$, and $P^\perp$ (called the LP set's \textbf{parameter subspace}) is the orthogonal complement of $P$. A chemical kinetic system $(\mathscr{N},K)$ is \textbf{positive equilibria log-parametrized (PLP) system} if its set of positive equilibria is an LP set, i.e., $E_+(\mathscr{N},K)=E(P_E,x^*)$ where $P_E$ is the flux subspace and $x^*$ is a given positive equilibrium.
\end{definition}

The \textit{Species Hyperplane Criterion} for absolute concentration robustness is recalled below.

\begin{theorem}[Theorem 3.12, \cite{LLMM2022}] 
If $(\mathscr{N},K)$ is a PLP system, then it has ACR is a species $A$ if and its parameter subspace $P_E^\perp$ is a subspace of the species hyperplane $\{x \in  \mathbb{R}^\mathscr{S} \mid x_A =0 \}$.
\end{theorem}

\begin{remark}
The flux subspace of the Anderies system is its \textbf{kinetic flux subspace}, denoted by $\widetilde{S}$. This subspace is the kinetic analogue of the stoichiometric subspace. If the stoichiometric subspace is the span of the reaction vectors, the kinetic flux subspace is the span of the fluxes in terms of the kinetic vectors. In light of the discussion above, if the vector $x^*$ is any positive steady state of the system, the set of positive equilibria consists of vectors $x$ such that the vector $\log (x) - \log (x^*)$ resides in $\widetilde{S}^\perp$. Specifically, for the Anderies system, $\widetilde{S}^\perp = \text{span } \{ [-Q, 1, 1]^\top \}$ where $Q = \dfrac{q_2- q_1}{p_2-p_1}$ \cite{FORMEN2023}.
\end{remark}

\end{document}